\documentclass[11pt,a4paper,twoside,bibliography=totoc]{scrartcl}
\usepackage{a4wide}
\usepackage{amsfonts}
\usepackage{amsmath}
\usepackage{mathtools}
\usepackage{amssymb}
\usepackage{array}
\usepackage{amsthm}
\usepackage{dsfont}
\usepackage[utf8]{inputenc}
\usepackage{caption}
\usepackage{subcaption}
\usepackage{xifthen}
\usepackage[urlcolor=blue,colorlinks=false]{hyperref}
\usepackage[ruled,vlined]{algorithm2e}
\usepackage{graphicx}
\usepackage{color}
\usepackage{pgf}
\usepackage{multirow}
\usepackage{multicol}
\usepackage{tabularx}
\usepackage{diagbox}
\usepackage{longtable}
\usepackage{listings} 
\usepackage{cleveref}
\usepackage{tikz}
\usetikzlibrary{math}
\usetikzlibrary{patterns}
\usepackage{graphicx}
\usepackage{wrapfig}

\newcommand{\N}{\ensuremath{\mathbb{N}}}

\newcommand{\T}{\ensuremath{\mathbb{T}}}

\newcommand{\Z}{\ensuremath{\mathbb{Z}}}

\newcommand{\R}{\ensuremath{\mathbb{R}}}
\newcommand{\C}{\ensuremath{\mathbb{C}}}

\newcommand{\ii}{\mathit{i}}
\newcommand{\e}{\textnormal{e}}

\newcommand{\eip}[1]{\textnormal{e}^{2\pi\ii{#1}}}
\newcommand{\eim}[1]{\textnormal{e}^{-2\pi\ii{#1}}}
\newcommand{\norm}[1]{\left\Vert #1\right\Vert}

\newcommand{\wrapm}[1]{\left| #1 \right|_{\mathbb{T}^d}}
\newcommand{\floor}[1]{\left\lfloor#1\right\rfloor}
\newcommand{\ceil}[1]{\left\lceil#1\right\rceil}

\newcommand{\pmat}[1]{\begin{pmatrix} #1 \end{pmatrix}}

\newcommand{\mat}[1]{\ensuremath \mathbf{\boldsymbol{#1}}}
\newcommand{\set}[1]{\left\{ #1 \right\}}
\newcommand{\setcond}[2]{\left\{#1 \mid #2\right\}}

\newcommand{\abs}[1]{\left| #1 \right |}
\renewcommand{\vec}[1]{\ensuremath \mathbf{\boldsymbol{#1}}}

\newcommand{\dirich}[1]{\ifthenelse{\isempty{#1}}
	{d_n}
	{d_n\left(#1\right)} }		
\newcommand{\dirichd}[1]{\ifthenelse{\isempty{#1}}
	{d_n'}
	{d_n'\left(#1\right)} }	
\newcommand{\dirichdd}[1]{\ifthenelse{\isempty{#1}}
	{d_n''}
	{d_n''\left(#1\right)} }	
\newcommand{\dirichm}[1]{\ifthenelse{\isempty{#1}}
	{\tilde d_n}
	{\tilde d_n\left(#1\right)} }
\newcommand{\comp}[2]{\left(#1\right)_{#2}}

\newcommand{\vnu}{{\vec{\nu}}}
\newcommand{\mInd}{\vnu}
\newcommand{\mIndSymb}{\nu}

\newcommand{\nodeSep}{q}

\renewcommand{\mathbf}[1]{\ensuremath{\boldsymbol{#1}}}

\newcommand{\diag}{\operatorname{diag}}

\DeclareMathOperator*{\smax}{\sigma_{\textnormal{max}}}
\DeclareMathOperator*{\smin}{\sigma_{\textnormal{min}}}

\DeclareMathOperator{\cond}{cond}

\renewcommand{\d}{\,\mathrm{d}}

\newtheorem{thm}{Theorem}[section]
\newtheorem{lemma}[thm]{Lemma}
\newtheorem{remark}[thm]{Remark}
\newtheorem{definition}[thm]{Definition}
\newtheorem{example}[thm]{Example}
\newtheorem{corollary}[thm]{Corollary}
\newtheorem{proposition}[thm]{Proposition}

\numberwithin{equation}{section}
\numberwithin{table}{section}
\numberwithin{figure}{section}

\newcommand{\bend}{\hspace*{0ex} \hfill \hbox{\vrule height
    1.5ex\vbox{\hrule width 1.4ex \vskip 1.4ex\hrule  width 1.4ex}\vrule
    height 1.5ex}}

\long\def\symbolfootnote[#1]#2{\begingroup%
\def\thefootnote{\fnsymbol{footnote}}\footnote[#1]{#2}\endgroup}

\crefname{lemma}{Lemma}{Lemmata}
\crefname{definition}{Definition}{Definitions}
\crefname{theorem}{Theorem}{Theorems}
\crefname{thm}{Theorem}{Theorems}
\crefname{corollary}{Corollary}{Corollaries}
\crefname{equation}{}{}
\crefname{remark}{Remark}{Remarks}
\crefname{algorithm}{Algorithm}{Algorithms}
\crefname{chapter}{Chapter}{Chapters}
\crefname{section}{Section}{Sections}
\crefname{table}{Table}{Tables}
\crefname{figure}{Figure}{Figures}
\crefname{example}{Example}{Examples}
\crefname{appendix}{Appendix}{Appendices}

\renewcommand{\thefootnote}{\fnsymbol{footnote}}

\allowdisplaybreaks
\title{Multivariate Vandermonde matrices with separated nodes on the unit circle are stable}

\date{\today}

\author{Stefan Kunis\footnotemark[1]\ \footnotemark[2] \qquad Dominik Nagel\footnotemark[1] \qquad Anna Strotmann\footnotemark[1]}

\newif\ifshow
\showtrue

\begin{document}
\maketitle

\begin{abstract}

	We prove explicit lower bounds for the smallest singular value and upper bounds for the condition number of rectangular, multivariate Vandermonde matrices with scattered nodes on the complex unit circle.
	Analogously to the Shannon-Nyquist criterion, the nodes are assumed to be separated by a constant divided by the used polynomial degree.
	If this constant grows linearly with the spatial dimension, the condition number is uniformly bounded.
	If it grows only logarithmically with the spatial dimension, the condition number grows slightly stronger than exponentially with the spatial dimension.
	Both results are quasi optimal and improve over all previously known results of such type.
	
	
	\medskip
	
\end{abstract}

\footnotetext[1]{
	Osnabr\"uck University, Institute of Mathematics
	\texttt{\{skunis,dnagel,astrotmann\}@uos.de}
	}

\footnotetext[2]{
	Osnabr\"uck University, Research Center of Cellular Nanoanalytics
	}

\section{Introduction}
\label{sec:introduction}

The condition number of rectangular Vandermonde matrices with nodes on the complex unit circle became important for the stability analysis of subspace methods like the Matrix Pencil method \cite{HuSa90}, ESPRIT \cite{RoKa1989} and MUSIC \cite{Sc1986}, see also \cite{StMo2005}.
A deterministic performance analysis is provided in \cite{AuBo2016, LiLiFa2020, Mo2015, ChTy2020}
and relies on the smallest and the largest singular values of such Vandermonde matrices.

In the univariate case, the condition number and the extremal singular values were studied intensively during the last years.
If nodes are on the unit circle and well-separated, tight upper bounds for the largest and lower bounds for the smallest singular value are proven in \cite{Mo2015,AuBo2019,Di2019} by means of extremal functions.
Furthermore, the situation in which nodes build clusters is investigated in \cite{BaDeGoYo2020, LiLi2021, BaDiGoYo2021, KuNa2020_1, Di2019, KuNa2020_2}.
In the multivariate case, only few results are available: The matrix in question has full rank if the normalized node separation scales with the square root of the spatial dimension \cite{KoLo2007,PoTa20132} or with the logarithm of the spatial dimension \cite{KuMoPeOh2017}, respectively.
Some quantitative results are available in \cite{KuPo2007,KuNa2020_1} under a linear scaling in the spatial dimension.

In this paper, we present new bounds for the extreme singular values of multivariate rectangular Vandermonde matrices with well-separated nodes on the complex unit circle.
The main result is a lower bound for the smallest singular value under a separation condition that is logarithmically scaling with the dimension. This makes the result in \cite{KuMoPeOh2017} quantitative, the bound itself decays slightly stronger than exponential in the spatial dimension, see \cref{thm:inghamBound}.
A second result refines \cite{KuPo2007} and provides a dimension independent lower bound under a linearly scaling assumption on the separation, see \cref{thm:lowerBoundKernel}.
\cref{tab:comparisonMultivariateWellSep} anticipates the results and simplifies the comparison to already existing ones.


\section{Preliminaries and previous results}
\label{sec:prelAndPrevRes}
Throughout the paper, $d\in\N$ always denotes the dimension and $\T:= \R/\Z= [0,1)$ the torus parametrizing the complex unit circle  $\setcond{z\in\C}{\abs{z}=1}=\setcond{\eip{t}\in\C}{t\in \T}$.
The \emph{wrap-around distance} between two nodes $\vec{t}, \vec{t}' \in \T^d$ is defined by
	\begin{equation*}
		\wrapm{\vec{t}-\vec{t}'}
		:= \min_{\vec{r}\in \Z^d} \norm{\vec{t}-\vec{t}'+\vec{r}}_\infty.
	\end{equation*}
	We note that this distance is the largest wrap-around distance in the coordinate directions.
	The \emph{minimal separation distance} of a node set $\Omega=\set{\vec{t}_1,\dotsc,\vec{t}_M}\subset \T^d$ of cardinality $M\in\N$ is given by
	\begin{equation*}
		q
		:= \min_{\vec{t}\neq\vec{t}'\in\Omega} \wrapm{\vec{t}-\vec{t}'}.
	\end{equation*}
For some parameter $N\in\N$, we call the node set \emph{well-separated} if the minimal separation distance $\nodeSep$ fulfills $qN > 1$, see \cref{fig:wellSepMultivariat} for an illustration.

Now set $\vec{z}_j := (\eip{\comp{\vec{t}}{1}},\dotsc,\eip{\comp{\vec{t}}{d}})^\top\in\C^d$, $j=1,\dotsc,M$, and let
$\mInd:=(\mIndSymb_1,\dotsc,\mIndSymb_d)^\top\in\N^d$ be a multi-index.
We are interested in the multivariate, rectangular Vandermonde matrix
\begin{equation}\label{eq:vandermondeDefMulti}
	\mat{A}	
	:=\mat{A}_N(\Omega)
	:= \pmat{\vec{z}_j^\mInd}_{\substack{ j=1,\dots,M \\ \mInd\in\N^d,\,\norm{\mInd}_\infty < N}}\in \C^{M\times N^d},
	\qquad \vec{z}_j^\mInd
	:=(\vec{z}_j)_1^{\mIndSymb_1}\cdot\dotsb\cdot(\vec{z}_j)_d^{\mIndSymb_d}.
\end{equation}
and its condition number $\cond(\mat{A})	:= {\smax(\mat{A})}/{\smin(\mat{A})}$
where $\smax(\mat{A})$ and $\smin(\mat{A})$ are the respective largest- and smallest singular values.
A necessary condition for $\smin(\mat{A})>0$ and hence having a finite condition number is that the nodes are distinct (in general, this condition is sufficient if and only if $d=1$).
Furthermore, the continuity of the smallest singular value with respect to the entries in $\mat{A}$ leads to $\lim_{q\to 0}\smin(\mat{A})=0$. Similarly, one obtains $\lim_{N\to\infty} \cond(\mat{A})=1$.
We continue with the collection of some known and easy to prove results in order to give a short overview.
\goodbreak
\begin{thm}\label{thm:condWell}
	With the above notation, we have the following results:
	\begin{enumerate}
	\item Without any further conditions, we have
	\begin{equation*}
		\smin(\mat{A})
		\le N^{\frac{d}{2}}
		\le \smax(\mat{A})
	\end{equation*}
	with equality if and only if for each pair of distinct nodes $\vec{t},\vec{t}'\in\Omega$ there exists a component $1\le s\le d$ such that $N\comp{\vec{t}-\vec{t}'}{s} \in \Z\setminus\set{0}$.
	\item If $qN>1$, then
	\begin{align*}
	  \left(N-\frac{1}{q}\right)^{\frac{1}{2}} \le &\smin(\mat{A})\text{ for }d=1
	  \text{ and }
	  \smax(\mat{A})\le \left(N+\frac{1}{q}\right)^{\frac{d}{2}} \text{ for }d\ge 1.
	\end{align*}
	\item Finally, if $\Omega=\frac{1}{M}\{0,\dotsc,M-1\}^d$ is a set of \emph{equispaced nodes} with $q=\frac{1}{M}$ and $qN\ge 1$, then
	\begin{equation*}
		\left(N-\frac{1}{q}\right)^{\frac{d}{2}}
		\le N^{\frac{d}{2}}\left(\frac{\floor{Nq}}{Nq}\right)^{\frac{d}{2}}
		= \smin(\mat{A})\le
		\smax(\mat{A})
		=  N^{\frac{d}{2}}\left(\frac{\ceil{Nq}}{Nq}\right)^{\frac{d}{2}}
		\le \left(N+\frac{1}{q}\right)^{\frac{d}{2}}
	\end{equation*}
	and the upper and lower bounds are tight for $q\searrow r/N$ and $q\nearrow r/N$, $r\in\N$, respectively.
	\end{enumerate}
\end{thm}
\begin{proof}
 The inequalities in the first result follow from the diagonal entries of $\mat{A}\mat{A}^*$ being $N^d$.
 Equality is equivalent to $\mat{A}\mat{A}^* = N^d\mat{I}_M$ ($\mat{I}_M$ denotes the identity matrix of size $M$) and direct computation shows that these off diagonal are all zero if and only for any distinct pair of nodes $\vec{t},\vec{t}'$ there is at least one component of $\vec{t}-\vec{t}'$ being in $\Z/N \setminus \set{0}$. See \cref{fig:nodesGridQuasiGrid} for two examples and \cite{BeFe2007} for the result when $d=1$.
 
 The second result is due to \cite{Mo2015,AuBo2019} for $d=1$ and uses extremal minorant and majorant functions, respectively.
 A tensor product majorant can be used similarly to \cite[Appendix A]{Li2015} to prove the upper bound when $d>1$.
 The third result can be found in \cite[Cor.~4.11]{KuPo2007}.
\end{proof}

\begin{figure}[ht]
	\begin{subfigure}{0.32\textwidth}
	\centering
	    \includegraphics[width=0.6\linewidth]{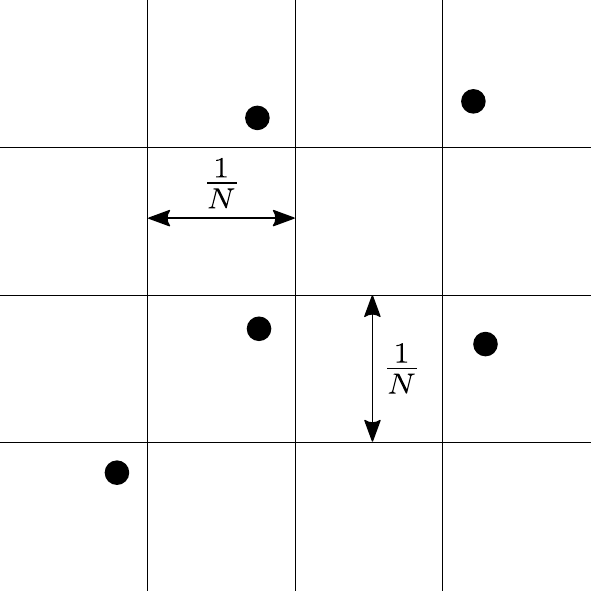}
	\end{subfigure}
	\begin{subfigure}{0.32\textwidth}
	\centering
		\includegraphics[width=0.6\linewidth]{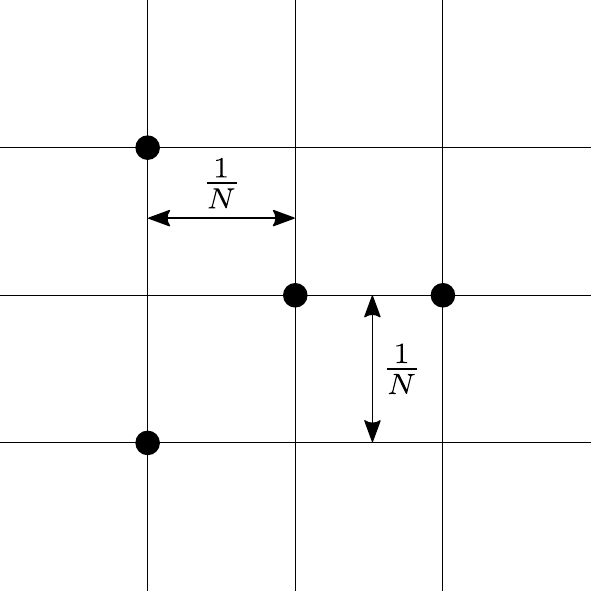}
	\end{subfigure}
	\begin{subfigure}{0.32\textwidth}
	\centering
		\includegraphics[width=0.6\linewidth]{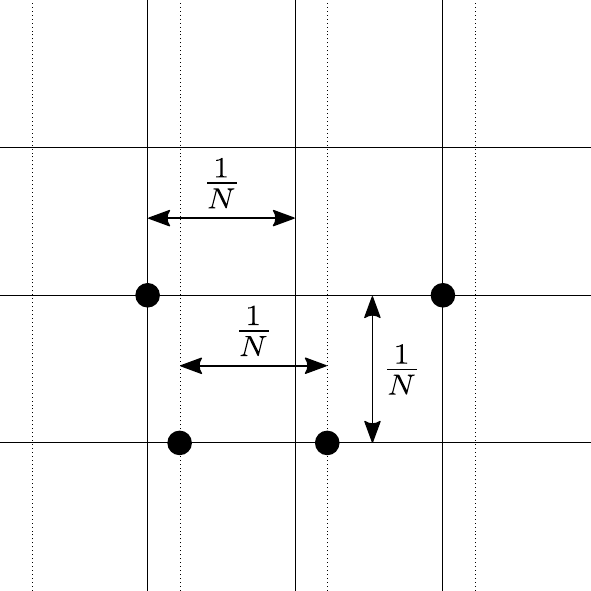}
	\end{subfigure}
	\caption{Left:~Well-separated node set in $\T^2$. Middle and right:~Examples of node sets with perfectly conditioned Vandermonde matrix for $d=2$; left:~nodes on a grid with width $1/N$;~right:~each two distinct nodes have one coordinate direction with separation $k/N, 0\neq k\in\Z$.}
	\label{fig:nodesGridQuasiGrid}
	\label{fig:wellSepMultivariat}
\end{figure}

\cref{thm:condWell} ii) and iii) show that the behaviour of the largest singular value with respect to the separation distance is completely understood for $qN>1$.
The very same is true for the smallest singular value only in the univariate situation $d=1$. 
The multivariate situation is much more involved.
\cref{tab:comparisonMultivariateWellSep} anticipates the result and simplifies the classification of the results and the comparison.

\begin{table}[ht]
\centering
\begin{tabular}{l||c|c}
 & $q(N-1) \ge $ & $\smin(\mat{A})(N-1)^{-\frac{d}{2}} \ge $\\ 
	\hline\hline
	\cref{thm:lowerBoundKernel} using \cite[Cor.~4.7]{KuPo2007} & $4d$& $0.9$\\
	\hline
	\cite[Ex.~4.4]{KuNa2020_1} and \cref{thm:wellSeparatedFromMulti}& $6d$ &
	$\frac{1}{3}d^{-d/4}$\\
	\hline
	 \cite[Cor.~3.3]{PoTa20132} using \cite{KoLo2007} & $\sqrt{d}$& is positive\\
	\hline
	\cref{rem:r123} with $r=1$ & $1.01\sqrt{d+2}$& $\frac{1}{2} d^{-d/4}$\\
	\hline
	\cite[Cor.~2.10]{KuMoPeOh2017} & $3+2\log d$ & is positive\\
	\hline
	\cref{thm:inghamBound} & $4.5+2.6\log d$ & $5.6^{-d}\left(1+\log d\right)^{-d/4}$
\end{tabular}\caption{Bounds for the smallest singular value of multivariate Vandermonde matrices with well-separated nodes showing the dependency on the spatial dimension $d$.} 
\label{tab:comparisonMultivariateWellSep}
\end{table}

\goodbreak

\section{Main result via a localizing function}
\label{sec:mainResult}

The following function was already used in \cite{KuMoPeOh2017} to ensure the full rank of $\mat{A}$ under the condition that the nodes are separated by $c_d/(N-1)$ where the constant $c_d>1$ is logarithmically dependent on the spatial dimension $d$.

\begin{definition}[{\cite{KuMoPeOh2017}}]
	\label{def:function}
		Let $d\in\mathbb{N}$ the dimension, $r\in\mathbb{N}$ and $p:=2r$.
		Furthermore, let $b,h >0 $ and $\varphi\colon \R\to\R$,
	\begin{equation*}
		\varphi(x)
		:=\begin{cases}
			\Big(1-\big(\frac{2x}{h}\big)^2\Big)^r,
			&\text{$\vert{x}\vert<\frac{h}{2}$,}\\
			0,
			&\text{else.}
			\end{cases}
	\end{equation*}
	We define the function $\psi\colon\R^d\to\R$,
	\begin{equation}
	\label{psi}
	\psi:=\bigg((2\pi b)^p-(-1)^r\sum_{s=1}^d\frac{\partial^p}{\partial x_s^p}\bigg)\bigotimes_{l=1}^d\varphi\ast\varphi.
	\end{equation}
\end{definition}


\begin{lemma}[{\cite[Lemma.~2.1]{KuMoPeOh2017}}]
	The function $\psi$ from \cref{def:function} has the following properties:
	\begin{enumerate}
	\item
		Its Fourier transform $\widehat{\psi}(y)=\int_{\R^d}\psi(x)\eim{y x}\d x$ is bounded and it holds
		\begin{equation*}
		\widehat{\psi}(y)
			\begin{cases}
			\ge 0,&\norm{y}_p\le b, \\
			\le 0,& \norm{y}_p \ge b,
			\end{cases}
		\end{equation*}
	\item
		its support is given by $supp(\psi)=[-h,h]^d$,
	\item $\psi(0)>0$, if $h>C_p\sqrt[p]{d}/b$ with $C_p\le\frac{2p+3}{e\pi}$.
	\end{enumerate}
\end{lemma}

While these properties were sufficient for proving $\smin(\mat{A})>0$, they do not allow to directly deduce a quantitative lower bound.
More advanced properties of $\psi$ are presented in the next lemma and utilized for the proof of an explicit lower bound on the smallest singular value in \cref{la:lowerBoundInghamConst} and \cref{thm:inghamBound}.

\begin{lemma}\label{la:evalsInZero}
	Let $\Gamma$ denote the gamma function. For the function $\psi$ from \cref{def:function}, we have
	\begin{equation*}
		\psi(0)
		= \Big(\frac{h\sqrt{\pi}(2r)!}{2\Gamma\big(2r+\frac{3}{2}\big)}\Big)^d (2\pi b)^{2r} -\Big(\frac{h\sqrt{\pi}(2r)!}{2\Gamma\big(2r+\frac{3}{2}\big)}\Big)^{d-1} \frac{d4^{2r}(r!)^2}{(2r+1)h^{2r-1}}
	\end{equation*}
	and
	\begin{equation*}
	\max_{v\in\mathbb{R}^d}\widehat{\psi}(v)
	=\widehat{\psi}(0)
	=(2\pi b)^p \bigg(\frac{h\sqrt{\pi}r!}{2\Gamma\big(r+\frac{3}{2}\big)}\bigg)^{2d}.
\end{equation*}
\end{lemma}

\begin{proof}
 The proof \cite[Lemma 2.1]{KuMoPeOh2017} already showed
	\begin{align*}
		\psi(0)
		&=\big(\varphi\ast\varphi\big)^{d-1}(0)\bigg((2\pi b)^p \frac{h\sqrt{\pi}p!}{2\Gamma\big(p+\frac{3}{2}\big)}-\frac{d4^p(r!)^2}{(p+1)h^{p-1}}\bigg).
	\end{align*}
	Together with
	\begin{align*}
		\varphi\ast\varphi(0)
		&=\frac{h}{2}\int_{-1}^1\big(1-y^2\big)^p\mathrm{d}y
		=\frac{h\sqrt{\pi}p!}{2\Gamma\Big(p+\frac{3}{2}\Big)},
	\end{align*}
	this yields the first result.
	The second claim follows from
	\begin{align*}
		\widehat{\psi}(v)
		=\bigg((2\pi b)^p-\sum_{s=1}^d(2\pi v_s)^p\bigg)\prod_{l=1}^d\big(\widehat{\varphi}(v_l)\big)^2
		\end{align*}
	and
	\begin{align*}
		\allowdisplaybreaks
		\big\vert{\widehat{\varphi}(v)}\big\vert
		&\le\int_{\mathbb{R}}\vert{\varphi(x)}\vert \mathrm{d}x
		=\widehat\varphi(0)
		=\frac{h}{2}\int_{-1}^1 (1-x^2)^r \mathrm{d}x
		=\frac{h\sqrt{\pi} r!}{2\Gamma\left(r+\frac{3}{2}\right)}.
	\end{align*}
\end{proof}

\begin{lemma}\label{la:lowerBoundInghamConst}
	For the function $\psi$ from \cref{def:function}, we have
	\begin{align*}
		\frac{\psi(0)}{\widehat{\psi}(0)}
		&> \left(\sqrt{\frac{2}{\pi}} \cdot \frac{\sqrt{r}}{h}\right)^d \left(1- d \frac{2\e^2}{\sqrt{\pi}} \sqrt{r} \left(\frac{2r}{\pi\e \cdot bh}\right)^{2r}\right).
	\intertext{Choosing $h= \frac{2p+3}{\e\pi} \sqrt[p]{d}\cdot \frac{1}{b}$ further yields}
		\frac{\psi(0)}{\widehat{\psi}(0)}
		&> \frac{1}{2} \left(\frac{4}{3} \cdot \frac{b}{\sqrt{p} \cdot\sqrt[p]{d}}\right)^d
	\intertext{and setting $p=2\ceil{\log(d)}$ finally leads to $h\le \frac{4\log(d)+7}{b\pi}$ and}
		\frac{\psi(0)}{\widehat{\psi}(0)}
		&> \frac{1}{2} \left(\frac{4}{3\sqrt{2}\e^2} \cdot \frac{b}{\sqrt{\log( d)+1}}\right)^d.
	\end{align*}
\end{lemma}

\begin{proof} We start from \cref{la:evalsInZero} and calculate
	\begin{align}
		\frac{\psi(0)}{\widehat{\psi}(0)}
		&=\frac{\Big(\frac{h\sqrt{\pi}(2r)!}{2\Gamma\big(2r+\frac{3}{2}\big)}\Big)^d(2\pi b)^{2r} -\Big(\frac{h\sqrt{\pi}(2r)!}{2\Gamma\big(2r+\frac{3}{2}\big)}\Big)^{d-1}\frac{d4^{2r}(r!)^2}{(2r+1)h^{2r-1}}}{(2\pi b)^{2r} \Big(\frac{h\sqrt{\pi}\cdot r!}{2\Gamma\big(r+\frac{3}{2}\big)}\Big)^{2d}}\nonumber\\
		&=\left(\frac{2(2r)!\Gamma^{2}\big(r+\frac{3}{2}\big)}{\Gamma\big(2r+\frac{3}{2}\big)h\sqrt{\pi}(r!)^{2}}\right)^d\bigg(1-\frac{d 2^{2r+1}\Gamma\big(2r+\frac{3}{2}\big)(r!)^2}{(2r)!h^{2r}\sqrt{\pi}(2r+1)\pi^{2r}b^{2r}}\bigg).
		\label{eq:lowerBoundInghamConst1}
	\end{align}
	The Gautschi-Wendel inequality gives
	\begin{equation}\label{eq:lowerBoundInghamConst2}
		\sqrt{x-\frac{1}{2}}
		< \frac{\Gamma(x+\frac{1}{2})}{\Gamma(x)}
		< \sqrt{x+\frac{1}{2}}.
	\end{equation}
	for $x>1/2$.
	In combination with the definition of the Gamma function, we obtain for the expression in the first bracket of \cref{eq:lowerBoundInghamConst1}
	\begin{equation}
	\begin{split}
		\frac{2(2r)!\Gamma^{2}\big(r+\frac{3}{2}\big)}{\Gamma\big(2r+\frac{3}{2}\big)h\sqrt{\pi}(r!)^{2}}
		&=\frac{2}{h\sqrt{\pi}} \cdot \frac{\Gamma(2r+1)}{\Gamma\big(2r+\frac{3}{2}\big)} \cdot \left(\frac{\Gamma\big(r+\frac{3}{2}\big)}{\Gamma(r+1)}\right)^{2}
		> \frac{2}{h\sqrt{\pi}} \cdot\frac{r+\frac{1}{2}}{\sqrt{2r+\frac{3}{2}}}\\
		&= \frac{1}{h\sqrt{\pi}} \cdot \frac{2r+1}{\sqrt{2r+\frac{3}{2}}}
		> \frac{\sqrt{2r+\frac{1}{2}}}{h\sqrt{\pi}}
		> \sqrt{\frac{2}{\pi}} \cdot \frac{\sqrt{r}}{h}.
	\end{split}
	\end{equation}
	Again with \cref{eq:lowerBoundInghamConst2} and Stirling's formula, the second summand in the right bracketed term of \cref{eq:lowerBoundInghamConst1} can be simplified to
	\begin{equation}\label{eq:lowerBoundInghamConst3}
	\begin{split}
		\frac{2^{2r+1}d\Gamma\big(2r+\frac{3}{2}\big)(r!)^2}{(2r)!h^{2r}\sqrt{\pi}(2r+1)\pi^{2r}b^{2r}}
		&< \frac{2d}{\sqrt{\pi}}\left(\frac{2}{\pi b h}\right)^{2r} \frac{\sqrt{2r+\frac{3}{2}}}{2r+1}(r!)^2\\
		&\le \frac{2d}{\sqrt{\pi}}\left(\frac{2}{\pi b h}\right)^{2r} \frac{\sqrt{2r+\frac{3}{2}}}{2r+1} \e^2 r^{2r+1} \e^{-2r}\\
		&\le d \frac{2\e^2}{\sqrt{\pi}} \sqrt{r} \left(\frac{2r}{\pi\e \cdot bh}\right)^{2r}.
	\end{split}
	\end{equation}
	
	The choice $h= \frac{2p+3}{\e\pi} \sqrt[p]{d}\cdot \frac{1}{b}$ leads to
	\begin{equation}\label{eq:lowerBoundInghamConst4}
		1- d \frac{2\e^2}{\sqrt{\pi}} \sqrt{r} \left(\frac{2r}{\pi\e \cdot bh}\right)^{2r}
		= 1- d \frac{2\e^2}{\sqrt{\pi}} \sqrt{r} \left(\frac{2r}{(4r+3)\sqrt[2r]{d}}\right)^{2r}
		=1- \frac{2\e^2}{\sqrt{\pi}} \sqrt{r} \left(\frac{2r}{(4r+3)}\right)^{2r}
		\ge \frac{1}{2}
	\end{equation}
	and
	\begin{equation}\label{eq:lowerBoundInghamConst5}
		\sqrt{\frac{2}{\pi}} \cdot \frac{\sqrt{r}}{h}
		= \frac{\sqrt{2r \pi}b\e}{(4r+3)\sqrt[2r]{d}}.
	\end{equation}
	Finally, setting $r=\ceil{\log(d)}$ and combining \cref{eq:lowerBoundInghamConst3,eq:lowerBoundInghamConst4,eq:lowerBoundInghamConst5}, yields
	\begin{equation*}
		\frac{\psi(0)}{\widehat{\psi}(0)}
		> \frac{1}{2} \left(\frac{4}{3}\cdot  \frac{b}{\sqrt{2\log(d)+2}\cdot  d^{\frac{1}{2\log(d)+2}}}\right)^d
		> \frac{1}{2} \left(\frac{4b}{3\sqrt{2}\cdot\e^2\sqrt{\log(d)+1}}\right)^d.
	\end{equation*}
\end{proof}

\begin{thm}\label{thm:inghamBound}
	Let $\mat{A}$ be the Vandermonde matrix from \cref{eq:vandermondeDefMulti} with well-separated nodes that further satisfy
	\begin{equation*}
		q (N-1)
		\ge \frac{8\log(d)+14}{\pi}.
	\end{equation*}
	Then the smallest singular value of $\mat{A}$ is bounded by
	\begin{equation*}
		\smin(\mat{A})
		\ge \frac{1}{\sqrt{2}} \left(\frac{\sqrt{2}}{3\e^2} \cdot \frac{1}{\sqrt{\log( d)+1}}\right)^{\frac{d}{2}} (N-1)^{\frac{d}{2}}.
	\end{equation*}
\end{thm}

\begin{proof}
	We follow the proof of \cite[Cor.~2.5]{KuMoPeOh2017}.
	Since we have
	\begin{equation*}
		\mat{A}
		= \diag\left(\vec{z}_j^{\big(\ceil{\frac{N-1}{2}},\dotsc,\ceil{\frac{N-1}{2}}\big)}\right) \cdot \widetilde{\mat{A}}
		\quad \text{with}\quad 
		\widetilde{\mat{A}}
		:=\pmat{\vec{z}_j^{\vec{\nu}}}_{\substack{j=1,\dotsc,M\\ \vec{\nu} \in \set{-\ceil{\frac{N-1}{2}},\dotsc,\floor{\frac{N-1}{2}}}^d}}
	\end{equation*}
	and that singular values are unitary invariant, it holds $\smin(\mat{A})=\smin({\widetilde{\mat{A}}})$.
	Now, we use the function $\psi$ from \cref{def:function} with parameters $b=\frac{N-1}{2}$, $p=2\log(d)$ and $h=\frac{8\log(d)+14}{(N-1)\pi}$ to obtain the following estimate by using the properties of $\psi$ and the Poisson summation formula.
	Furthermore, notice that the condition on the separation distance $q$ says that we have a $h$-separated node set.
	For arbitrary $\vec{u}\in \C^M$, we have
	\begin{align*}
		\widehat{\psi}(0) \norm{\widetilde{\mat{A}}^*\vec{u}}^2
		&= \max_{\vec{x}\in\R^d} \widehat\psi(\vec{x}) \sum_{\vec{\nu} \in \set{-\ceil{\frac{N-1}{2}},\dotsc,\floor{\frac{N-1}{2}}}^d} \left|\sum_{j=1}^M \comp{\vec{u}}{j} \eip{\vec{\nu}^* \vec{t}_j}\right|^2\\
		&\ge
		\sum_{\vec{\nu}\in\Z^d} \widehat \psi(\vec{\nu}) \left|\sum_{j=1}^M \comp{\vec{u}}{j} \eip{\vec{\nu}^* \vec{t}_j}\right|^2\\
		&= \sum_{j=1}^M \sum_{\ell=1}^M \comp{\vec{u}}{j} \overline{\comp{\vec{u}}{\ell}} \sum_{\vec{r}\in\Z^d} \psi(\vec{t}_j-\vec{t}_{\ell} + \vec{r})
		=\psi(\vec{0}) \sum_{j=1}^M \left|\comp{\vec{u}}{j}\right|^2
		=\psi(\vec{0}) \norm{\vec{u}}^2.
	\end{align*}	
	Using the variational characterization of the smallest singular value leads to
	\begin{equation*}
		\smin(\mat{A})^2
		= \min_{\vec{u}\in \C^{M}\setminus\set{0}} \frac{\norm{\widetilde{\mat{A}}^*\vec{u}}^2}{\norm{\vec{u}}^2}
		\ge \frac{\psi(0)}{\widehat{\psi}(0)}.
	\end{equation*}
	Applying \cref{la:lowerBoundInghamConst} with the substitution of $b$ with $(N-1)/2$ yields the result.
\end{proof}

\begin{remark}\label{rem:r123}
    Under the assumed scaling $qN=a\log d +b \nearrow s$, $s\in\N$, $a>0$, $b\in\R$, \cref{thm:condWell} iii) yields the almost matching bound
    \begin{equation*}
     N^{-d}\smin^2(\mat{A})\le c_1\left(1-\frac{1}{a\log d+b}\right)^d\le \exp\left(-\frac{c_2\cdot d}{\log d}\right)
    \end{equation*}
    for some absolute constants $c_1,c_2>0$.

	Using \cref{eq:lowerBoundInghamConst1} directly for special small values of $r$ leads to explicit expressions for $\psi(0)/\widehat{\psi}(0)$.
	After fixing $r$, the parameter $h$ is chosen such the term becomes maximal,
	\begin{equation*}
		h\cdot b
		=\begin{cases}
			\frac{\sqrt{5}}{\sqrt{2}\pi}\sqrt{d+2}, & r=1,\\
			\frac{\sqrt{3}}{\pi}\left(\frac{7}{2}\right)^{\frac{1}{4}}(d+4)^{\frac{1}{4}} , & r=2,\\
			\frac{\sqrt{3}\cdot 143^{1/6}}{2^{1/3}\pi}(d+6)^{\frac{1}{6}}, & r=3
		\end{cases}
	\end{equation*}
	and we finally obtain
	\begin{equation*}
		\frac{b^{-d} \psi(0)}{\widehat{\psi}(0)}
		=\begin{cases}
			2^{\frac{3 d}{2}+1} 5^{-\frac{3 d}{2}} (d+2)^{-\frac{d}{2}-1} (3 \pi )^d, & r=1,\\
			2^{\frac{5 d}{4}+2} 3^{-\frac{d}{2}} 7^{-\frac{5 d}{4}} (d+4)^{-\frac{d}{4}-1} (5 \pi )^d, & r=2,\\
			2^{\frac{7 d}{3}+1} 3^{1-\frac{3 d}{2}} 143^{-\frac{7 d}{6}} (d+6)^{-\frac{d}{6}-1} (175 \pi )^d, & r=3.
		\end{cases}
	\end{equation*}
	Resulting bounds for the smallest singular of $\mat{A}$ and the conditions on the separation $q$ from the choice of $h$, can be obtained analogously to the proof of \cref{thm:inghamBound} and are given in \cref{tab:values} for the first dimensions.

	\begin{table}[h]
	\centering
	\begin{tabular}[h]{|l|c|c|c|c|c|c|}
		\hline
		& \multicolumn{3}{c}{$q(N-1)\ge$}\vline & \multicolumn{3}{c}{$\smin(\mat{A}_N) (N-1)^{-\frac{d}{2}}\ge$}\vline \\
		\hline
		\diagbox{$r$}{$d$} &$1$ & $2$ & $3$ & $1$ & $2$ & $3$ \\
		\hline
		$1$ &$1.744$& $2.014$&  $2.251$ &$0.677$  & $0.421$ & $0.246$\\
		\hline
		$2$ & $2.256$ & $2.361$& $2.454$& $0.711$  & $0.494$  & $0.335$\\
		\hline
		$3$ & $2.769$ & $2.831 $& $2.887$ & $0.710$  & $0.499$  & $ 0.347$\\
		\hline
	\end{tabular}
	\caption{Explicit constant of the separation condition and resulting bounds for the smalles singular value.}\label{tab:values}
	\end{table}
\end{remark}

\goodbreak
\section{Further results}
\label{sec:furtherResults}

The following result for the case of well-separated nodes was already given as a special case of a result for multivariate clustered node configurations.

\begin{thm}[{Lower bound on the smallest singular value,\cite[Ex.~4.4]{KuNa2020_1}}]
\label{thm:wellSeparatedFromMulti}
	Let $M,d, N\in \N$, $N>\max\set{M,2(d+2)^2}$, and $\mat{A} \in \C^{M\times N^d}$ be a Vandermonde matrix as in \cref{eq:vandermondeDefMulti} with separation distance satisfying
	\begin{equation*}
		qN>6d,
	\end{equation*}
	then we have
	\begin{equation*}
		\smin(\mat{A})
		\ge \frac{N^{d/2}}{3d^{d/4}}.
	\end{equation*}
\end{thm}

Finally, we present a lower bound on the smallest singular value that is proven by means of fast decaying trigonometric kernel functions in combination with the Gershgorin circle theorem, see \cite{KuPo2007}.

\begin{thm}\label{thm:lowerBoundKernel}
 Let $M,d, N\in \N$, $d\ge 2$, $N$ being even and $\mat{A} \in \C^{M\times N^d}$ be a Vandermonde matrix as in \cref{eq:vandermondeDefMulti} with separation distance satisfying
 \begin{equation*}
		qN>4d,
 \end{equation*}
 then we have
 \begin{equation*}
  \smin(\mat{A})> 0.9 N^{d/2}.
 \end{equation*}
\end{thm}
\begin{proof}
 With a slightly different normalization and an irrelevant shift, Corollary 4.7 in \cite{KuPo2007} gives
 \begin{equation*}
  \smin(\mat{A}\mat{D}\mat{A}^*) \ge 1-\left(\frac{2d}{Nq}\right)^{d+1}
 \end{equation*}
 for $Nq>2d$.
 The involved diagonal matrix $\mat{D} \in \R^{N^d\times N^d}$ has positive diagonal entries and the Rayleigh-Ritz characterization of the smallest eigenvalue of Hermitian matrices allows for the estimate
 \begin{align*}
		\smin(\mat{A}\mat{D}\mat{A}^*) 
		&=\min_{\vec{x}\in\C^{M},\norm{\vec{x}}=1} \norm{\mat{D}^{1/2}\mat{A}^* \vec{x}}^2\\
		&\le \min_{\vec{x}\in\C^{M},\norm{\vec{x}}=1} \norm{\mat{D}} \norm{\mat{A}^* \vec{x}}^2
		= \norm{\mat{D}}\smin(\mat{A}\mat{A}^*).
	\end{align*}
 The diagonal matrix is given by a sampled B-spline and satisfies
 \begin{equation*}
  \|\mat{D}\|\le \frac{1}{N^d\left(1-2\zeta(d+1)(2\pi)^{-d-1}\right)^d},
 \end{equation*}
 see \cite[Thm.~3.3 and the proof of Cor.~3.5]{KuPo2007}, where $\zeta$ denotes the Riemann zeta function.
 We finally use the assumption $Nq>4d$, Bernoulli's inequality, and monotony of Riemann's zeta function to estimate
 \begin{align*}
  \smin^2(\mat{A})&\ge N^d\left(1-2\zeta(d+1)(2\pi)^{-d-1}\right)^d \left(1-2^{-d-1}\right)\\
  &\ge N^d\left(1-2d\zeta(3)(2\pi)^{-d-1}\right) \left(1-2^{-d-1}\right)\\
  &>0.85 N^d.
 \end{align*}
\end{proof}

We close by noting that the result of \cref{thm:lowerBoundKernel} cannot be achieved under considerably weaker assumptions, since for $qN\in o(d)$, \cref{thm:condWell} iii) yields
    \begin{equation*}
     N^{-d}\smin^2(\mat{A})\le \left(1-\frac{1}{o(d)}\right)^d \to 0
    \end{equation*}
    for $d\to\infty$.
    
\bibliographystyle{abbrv}
\bibliography{references}

\end{document}